\documentclass[11pt,a4paper]{amsart}

\usepackage{amssymb}
\usepackage{amsmath,amsthm,geometry,enumerate}
\usepackage[dvips]{graphicx}
\usepackage{psfrag}
\usepackage{amscd}
\usepackage[all]{xy}
\DeclareMathOperator{\pic}{Pic}

\DeclareMathOperator{\Div}{Div}
\DeclareMathOperator{\typeI}{Type \ I}
\DeclareMathOperator{\typeII}{Type  \ II}

\newcommand{\Mm}[2]{\mathcal{M}_{#1,#2}}
\newcommand{\Mb}[1]{\overline{\mathcal{M}}_{#1}}
\newcommand{\Mmb}[2]{\overline{\mathcal{M}}_{#1,#2}}
\newcommand{\bitwo}{[\overline{\mathcal{B}}_2]_Q}
\newcommand{\bithree}{[\overline{\mathcal{B}}_3]_Q}
\newcommand{\invtwo}{[\overline{{I}}_6^{inv}]}
\newcommand{\invthree}{[\overline{{I}}_8^{inv}]}
\newcommand{\verma}{\pi_{1 2}^*([\overline{I}_6])}
\newcommand{\vermb}{\pi_{7 8}^*([\overline{I}_6])}

\newcommand{\Sym}[1]{\mathbb{S}_{#1}}

\newtheorem{theorem}{Theorem}
\newtheorem{proposition}[theorem]{Proposition}

\newtheorem{corollary}[theorem]{Corollary}
\newtheorem{lemma}[theorem]{Lemma}
\theoremstyle{definition}
\newtheorem{definition}{Definition}
\newtheorem{remark}{\textbf{Remark}}
\newtheorem{notation}[definition]{\textbf{Notation}}

\begin{document}

\title{The class of the bielliptic locus in genus $3$}
\author{Carel Faber}
\address{Department of Mathematics, KTH Royal Institute of Technology,
Lindstedtsv\"agen 25, 10044 Stockholm, Sweden.}
\email{faber@math.kth.se}
\address{Department of Mathematics, Utrecht University,
P.O. Box 80010, 3508 TA Utrecht, The Netherlands.}
\email{c.f.faber@uu.nl}
\author{Nicola Pagani}
\address{Institut f\"ur Algebraische Geometrie, Leibniz Universit\"at Hannover,
Welfengarten 1, 30167 Hannover, Germany.}
\email{npagani@math.uni-hannover.de}
\address{University of Liverpool,
Department of Mathematical Sciences,
Liverpool, L69 7ZL,
United Kingdom.}
\email{pagani@liv.ac.uk}

\begin{abstract} Let the bielliptic locus be the closure in the moduli space of stable curves of the locus of smooth curves that are double covers of genus $1$ curves. In this paper we compute the class of the bielliptic locus in $\Mb3$ in terms of a standard basis of the rational Chow group of codimension-$2$ classes in the moduli space.  Our method is to test the class on the hyperelliptic locus: this gives the desired result up to two free parameters, which are then determined by intersecting the locus with two surfaces in $\Mb3$.
\end{abstract}
\maketitle
\setcounter{tocdepth}{2}

\section {The main result}

A smooth bielliptic curve is a genus $g$ curve that is a double cover of a smooth genus $1$ curve, with $2g-2$ ramification points by the Riemann-Hurwitz formula. It follows that the locus of bielliptic curves has codimension $g-1$ in $\Mb{g}$, the moduli stack of stable curves of genus $g$. 
We are interested in the following general problem: given a family over a base of dimension $g-1$, how many bielliptic curves occur in it? We solve this problem by expressing the class of the bielliptic locus in terms of standard classes in the case of genus $3$ (and $2$).
Note that the bielliptic locus in genus $3$ is the main component of the
singular locus of the coarse moduli space $\overline{M}_3$.
The main result of this paper is the following.
\begin{theorem} \label{bithree} The rational stack class $\bithree$ of the bielliptic locus in $\Mb3$ equals
\begin{displaymath}
\bithree= \frac{2673}{2} \lambda^2- 267 \lambda \delta_0-651 \lambda \delta_1+ \frac{27}{2} \delta_0^2 + 69 \delta_0 \delta_1 + \frac{177}{2} \delta_1^2 - \frac{9}{2} \kappa_2.
\end{displaymath}
\end{theorem}

\begin{proof} In \cite{faber} the first author has studied the codimension-$2$ rational Chow group of $\Mb3$. In \cite[Thm.~2.9]{faber} it is proved that
\begin{equation} \label{basiscod2}
\lambda^2, \quad \lambda \delta_0, \quad \lambda \delta_1, \quad \delta_0^2, \quad \delta_0 \delta_1, \quad \delta_1^2, \quad \kappa_2
\end{equation}
is a basis for $A^2_{\mathbb{Q}}(\Mb3$).

We obtain the result by considering the pull-back via the map from the moduli stack of admissible hyperelliptic curves:
\begin{equation}
\phi\colon \overline{\mathcal{H}}_3^{adm} \to \Mb3.
\end{equation}
We prove in Proposition \ref{pullbackprop} that the pull-back $\phi^*(\bithree)$ is a multiple of $\invthree$: the class of the locus in $\overline{\mathcal{H}}_3^{adm}$ of curves admitting an involution that acts without fixed points on the set of $8$ Weierstrass points (the superscript refers to the invariance of this locus under the $\Sym8$-action permuting the Weierstrass points). In Section \ref{testsurfaces} we prove that in fact $\phi^*(\bithree)=\invthree$, by computing the class $\bithree$ on a suitable test surface $\Sigma_8$. 

The moduli stack $\overline{\mathcal{H}}_3^{adm}$ admits a natural map to $\overline{\mathcal{M}}_{0,8}$; we identify the Chow groups of the latter with those of the former via the pull-back map. In Proposition \ref{classinvmz8} we compute the class $\invthree$ in terms of boundary strata classes.

As shown in Section \ref{Pulling},
the linear map
\begin{displaymath}
\phi^*\colon A^2_{\mathbb{Q}}(\Mb3) \to A^2_{\mathbb{Q}}(\overline{\mathcal{M}}_{0,8})^{\Sym8}
\end{displaymath}
is surjective; thus, it has $1$-dimensional kernel, since the image has dimension $6$ (cf.~the beginning of Section \ref{invar}). To express the class $\bithree$ in the chosen basis \eqref{basiscod2}, we write in Lemma \ref{matrix} the matrix associated to the above linear map, where we have fixed the invariant boundary strata classes as a basis for the image. To conclude, we need to calculate the missing parameter coming from the kernel of $\phi^*$. This is done in Section \ref{testsurfaces}, by evaluating the class $\bithree$ on a test surface $\Sigma_1$ containing bielliptic nonhyperelliptic curves.
\end{proof} 
\noindent In Section \ref{testsurfaces}, we obtain 
as a corollary that
the degree of the bielliptic locus in the $\mathbb{P}^{14}$
parameterizing plane quartic curves equals $225$. 
Thus, we prove a classical enumerative
geometry result via the moduli space, in a sense in the spirit
of Mumford's paper \cite{M-Enum}; we know of no other way to obtain this
result.

We observe that an easy but nontrivial check of Theorem \ref{bithree} can be made on a suitable test surface $\Sigma_2$ where there are no bielliptic curves; this is done in Section~\ref{testsurfaces}.
However, we stress that the pull-back to the hyperelliptic locus
is essential for us; we are not able to prove our main result
by test surfaces alone.

Note also that we can compute the class of the bielliptic locus in $\Mb2$ with exactly the same method, with the simplifying difference that all genus $2$ curves are hyperelliptic. 
We obtain the following result:

\begin{proposition}
\label{prop2}
The stack class $\bitwo$ of the bielliptic locus in $\Mb2$ can be written as
\begin{displaymath}
\bitwo= 15 \lambda + 3 \delta_1= \frac{3}{2} \delta_0 + 6 \delta_1.
\end{displaymath}
\end{proposition}
\noindent This agrees with the result 
for the usual fundamental class
stated in \cite[p.~6]{faberbanach}.
\begin{proof} With the obvious adjustments of notation from the proof of the theorem above, we see in Proposition \ref{pullbackprop} that $\phi^*(\bitwo)=\invtwo$. The map $\phi^*$ is an isomorphism of $\pic_{\mathbb{Q}}(\overline{\mathcal{M}}_{2})$ with $\pic_{\mathbb{Q}}(\Mb{0,6})^{\Sym6}$, as recalled in \eqref{relationgenus0hyper}. The class $\invtwo$ is computed in \eqref{proofgenus2}.
\end{proof}

Throughout this paper we work with Chow groups with rational coefficients. 
We express our results in terms of the \emph{stack} classes. 

\smallskip
\emph{Acknowledgements.}
\thanks{This project was carried out at KTH Royal Institute of Technology.
The second author was supported by grant KAW 2005.0098 from the
Knut and Alice Wallenberg Foundation.
}

\section{Admissible double covers}

We begin by recalling admissible double covers. Admissible covers were introduced by Harris and Mumford in their seminal paper \cite{hamu}. The definition we give here is the one of \cite[Section 4.1]{acv}, adapted for the special case of degree-two covers.

\begin{definition} Let $C$ be a semistable curve of genus $g$. An \emph{admissible double cover} with source $C$ is the datum of a stable $n$-pointed curve $(C', x_1, \ldots, x_n)$ of genus $g'\le g$ and of a finite map $\phi\colon C \to C'$ of degree $2$ such that:
\begin{enumerate} 
\item the restriction $\phi^{sm}\colon C^{sm} \to C'^{sm}$ to the smooth locus is branched exactly at the marked points;
\item the image via $\phi$ of each node is a node.
\end{enumerate}
An \emph{admissible hyperelliptic structure} on $C$ is an admissible double cover where the genus $g'$ equals $0$, while an \emph{admissible bielliptic structure} corresponds to the case when $g'$ is $1$.
\end{definition}

By using the Riemann-Hurwitz formula and an induction on the number of nodes of $C'$, one can see that the number $n$ of marked points in the above definition must be $2g+2-4g'$. 
One can define families of admissible double covers, and isomorphisms of families of admissible double covers. In particular, we have the two moduli stacks $\overline{\mathcal{H}}_g^{adm}$ and $\overline{\mathcal{B}}^{adm}_g$ parameterizing hyperelliptic and bielliptic curves with a chosen admissible structure. They are both smooth proper Deligne-Mumford stacks, the first of dimension $2g-1$, the second of dimension $2g-2$. 

We will use the following two maps. To each family of admissible hyperelliptic covers one can associate the target family of stable genus $0$ curves together with the ordered branch divisor. This gives a map (a $\mu_2$-gerbe)
$
\overline{\mathcal{H}}_g^{adm} \to \Mmb0{2g+2}. 
$
Given a family of admissible bielliptic covers, one can forget all the extra structure besides the source family $C$ of semistable curves, and then contract all rational bridges (the rational components that intersect the closure of the complement in precisely two points). This gives a map
$
\overline{\mathcal{B}}_g^{adm} \to \Mb{g}.
$
It follows from the properness of $\overline{\mathcal{B}}_g^{adm}$ that on every family of stable curves, the locus corresponding to stable bielliptic curves forms a closed subscheme. We have thus a well-defined class \begin{displaymath}[\overline{\mathcal{B}}_g]_Q \in A^{g-1}_{\mathbb{Q}}(\Mb{g}).\end{displaymath}

\section{Loci in moduli spaces of stable pointed genus $0$ curves}
\label{loci}

The following loci in $\Mmb0{n}$ will play a central role. 

\begin{definition}
Let $n\ge3$ be an integer. Put $m=\lfloor n/2\rfloor$.
We define $I_{n}$ as the closed subscheme of $\Mm0{n}$ that parameterizes curves $(C,x_1, \ldots, x_{n})$ admitting an involution $\sigma$ whose induced permutation of the marked points is $(12)(34) \ldots (2m-1,2m)$. Let $\overline{I}_{n}$ be the closure of $I_{n}$ in $\Mmb0{n}$.
\end{definition}
Note that $\overline{I}_{n}$ has codimension $g-1$ for $n=2g+1$ or $2g+2$.
Let us now consider the invariant notion associated to the previous one.

\begin{definition} We define $I^{inv}_{n}$ as the closed subscheme in $\Mm0{n}$ that parameterizes curves $(C,x_1, \ldots, x_{n})$ such that $C$ admits
an involution that induces an action on the set of
marked points with at most one fixed point.
Let $\overline{I}^{inv}_{n}$ be the closure of $I^{inv}_{n}$ in $\Mmb0{n}$.
\end{definition}

Let us take the chance to fix the notation for certain boundary strata classes of $\Mmb0n$. 
\begin{notation} Given a partition $[n]=A_1 \sqcup A_2$ with $|A_i| \geq 2$, the general element of the divisor $(A_1|A_2)=(A_2|A_1)$ is made of genus $0$ curves with two irreducible components, one of them containing the marked points in $A_1$ and the other those in $A_2$. Given a partition $[n]=A_1 \sqcup A_2 \sqcup A_3$ with $|A_2| \geq 1$ and $|A_1|, |A_3| \geq 2$, the general element of the codimension-$2$ boundary stratum $(A_1|A_2|A_3)=(A_3|A_2|A_1)$ is made of genus $0$ curves with three irreducible components, the central one containing the marked points in $A_2$ and the extreme ones those in $A_1$ and $A_3$.
\label{notation}

We also fix the notation for the invariant boundary strata classes on $\Mmb0n$.
Given a partition $n= \lambda_1+ \lambda_2$ with $\lambda_i \geq 2$, the invariant divisor $d_{\lambda_1, \lambda_2}=d_{\lambda_2, \lambda_1}$ is the sum of the distinct divisors $(A_1|A_2)$ such that $|A_i|= \lambda_i$. Given a partition $n= \lambda_1+ \lambda_2+\lambda_3$ with $\lambda_2 \geq 1$ and $\lambda_1, \lambda_3 \geq 2$, the invariant codimension-$2$ boundary stratum $d_{\lambda_1, \lambda_2, \lambda_3}= d_{\lambda_3, \lambda_2, \lambda_1}$ is the sum of the distinct codimension-$2$ boundary strata $(A_1|A_2|A_3)$ such that $|A_i|= \lambda_i$. 
\end{notation}

\begin{figure}[ht]
\includegraphics[scale=0.4]{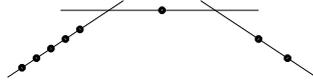}
\caption{The invariant stratum $d_{5,1,2}$ consists of $168$ strata
of this type.}
\label{d512}
\end{figure}

We now turn our attention to the genus $2$ case. Vermeire in \cite{vermeire} has computed 
\begin{equation} \label{verm}
[\overline{I}_6]=(15|2346)+ (25|1346)+ (36|1245) + (46|1235) - (56|1234) + 2 (125|346).
\end{equation}
From this, it is immediate to compute the class of $\overline{I}^{inv}_6$ in terms of the boundary divisors. Let $d_{2,4}$ and $d_{3,3}$ be the two invariant divisor classes in $\Mmb06$. The invariant divisor $\overline{I}^{inv}_6$ is the union of $15$ irreducible divisors, each of them corresponding to an element in $\Sym6$ in the conjugacy class of $(12)(34)(56)$. Now since $d_{2,4}$ is the sum of $15$ boundary divisor classes, and $d_{3,3}$ is the sum of the remaining $10$, we obtain the equality in $A^1(\Mmb06)$
\begin{equation} \label{proofgenus2}
\invtwo= 3 d_{2,4}+ 3 d_{3,3}.
\end{equation}

\section{The bielliptic locus and the invariant locus}

We consider the moduli space $\overline{\mathcal{H}}_g^{adm}$, which parameterizes admissible hyperelliptic curves. We have a diagram:
\begin{displaymath}\xymatrix{
\overline{\mathcal{H}}_g^{adm} \ar@/^15pt/^{\phi}[rr]\ar[d]^{\pi} \ar^j[r]& \overline{\mathcal{H}}_g \ar^{i}[r] & \Mb{g}\\ 
\Mmb0{2g+2}&
}
\end{displaymath}
The map $j$ forgets the structure of admissible double covers, and the representable map $i$ is a closed embedding. In particular, $j$ forgets the ordering on the ramification divisor and contracts the strictly semistable components. 
We will implicitly assume the isomorphism $\pi^*$ at the level of the Chow groups. 
 Note also that the pull-back $j^*$ is an isomorphism between the Chow groups of $\overline{\mathcal{H}}_g$ and the $\Sym{2g+2}$-invariants of the Chow groups of $\Mmb0{2g+2}$. 

\begin{proposition} \label{pullbackprop} The inverse images via $\phi$ of $\overline{\mathcal{B}}_2$ and $\overline{\mathcal{B}}_3$ are respectively $\overline{I}_6^{inv}$ and $\overline{I}_8^{inv}$. In other words, we have that $\phi^*(\bitwo)= \invtwo$, and there exists $\epsilon \in \mathbb{Z}_{>0}$ such that
\begin{displaymath}
\phi^*(\bithree)= \epsilon \cdot \invthree.
\end{displaymath}
\end{proposition}

\begin{proof} 

We study the case of genus $3$: the other case is similar and simpler. 

We start by proving that ${I}_8^{inv}=\phi^{-1}({\mathcal{B}}_3)$ and that $\overline{I}_8^{inv}$ is contained in  $\phi^{-1}(\overline{\mathcal{B}}_3)$. Let $C$ be a smooth genus $3$ curve with a hyperelliptic quotient map $\psi\colon C \to C'$.
If $C$ also admits a bielliptic involution, this descends to an involution of $C'$ because the hyperelliptic involution $\iota$ commutes with all automorphisms of $C$ (as it is the automorphism of order $2$ associated to the double cover given by the unique $g^1_2$).
Vice versa, an involution of $C'$ that preserves the branch divisor lifts to
an involution of the hyperelliptic curve if and only if it does not fix any
branch point.
From this, it follows that the bielliptic involution on
the branch locus of $\psi$ swaps the $8$ unordered points two-by-two.
This proves the equality ${I}_8^{inv}=\phi^{-1}({\mathcal{B}}_3)$;
the inclusion
$\overline{I}_8^{inv}\subseteq \phi^{-1}(\overline{\mathcal{B}}_3)$ follows
immediately.

Since the pull-back of $\bithree$ is $\Sym8$-invariant,
we conclude the proof by verifying two claims. (Note that $\epsilon$ is
the intersection multiplicity of $\overline{\mathcal{B}}_3$ and
$\overline{\mathcal{H}}_3$ along $\phi(\overline{I}_8^{inv})$; in genus~$2$,
the corresponding intersection is trivially transversal.)
Recall that $\overline{I}_{n}$ has codimension $g-1$ for $n=2g+1$ or $2g+2$ (cf.~the beginning of Section \ref{loci}).

{\bf Claim 1.}
In each of the three irreducible boundary divisors of $\overline{\mathcal{H}}_3$, consider the open locus of curves that have the minimum number of singular points. On each of these open loci, the condition of having a bielliptic structure cuts out a locus of codimension \emph{strictly greater than one}.
(By the above, one needs to check that none of the three types
$d_{6,2}$, $d_{5,3}$, and $d_{4,4}$
of stable unordered $8$-pointed curves of genus $0$ corresponding to the
codimension-$1$ boundary strata in $\Mmb08/{\Sym8}$ 
admits an involution exchanging the marked points two-by-two
in codimension at most one. The type $d_{5,3}$ never admits such an involution.
The type $d_{6,2}$ admits such an involution in codimension two, as one sees
by considering a $7$-tuple on $\mathbb{P}^1$. For $d_{4,4}$, if the components
are fixed, then the $5$-tuples on either side must be special; if the
components are exchanged, the locus is the diagonal in
$\Mm05\times\Mm05$.)

{\bf Claim 2.}
None of the six codimension-$2$ boundary strata classes of $\overline{\mathcal{H}}_3$ admits a bielliptic structure generically.
(By the above, one needs to check that none of the six types
of stable unordered $8$-pointed curves of genus $0$ corresponding to the
codimension-$2$ boundary strata in $\Mmb08/{\Sym8}$ (listed in
(\ref{basiscod2hyperell}) below) generically admits an involution exchanging
the marked points two-by-two.
For $d_{5,1,2}$, $d_{4,2,2}$, $d_{4,1,3}$, and $d_{3,3,2}$, one sees directly
that the components and hence the nodes must be fixed. In the case
$d_{4,2,2}$, the $j$-invariant of the unordered $4$-pointed
curve in the middle must be special; the other cases are obvious.
For $d_{3,2,3}$ and $d_{2,4,2}$, a similar argument applies when the components
are fixed. If the outer components are exchanged, they must have the same
$j$-invariant in the case $d_{3,2,3}$, whereas one deals with the case
$d_{2,4,2}$ by observing that a general $6$-tuple on $\mathbb{P}^1$ does not
admit a fixed-point-free involution.)
\end{proof}

We will eventually be able to prove that $\epsilon=1$ in Section \ref{testsurfaces}, by enumerating bielliptic curves on the test surface $\Sigma_8$. We remark that when the genus is higher than $3$, there are no smooth bielliptic-hyperelliptic curves. (If $g$ is the genus of such a curve $C$, the inequality
$g\ge 2(g-1)-1$ holds, since $C$ also doubly covers a curve of genus $g-1$.)
Determining stable hyperelliptic curves that admit a bielliptic involution can therefore produce several relations, inductively in $g$, among the coefficients of $[\overline{\mathcal{B}}_g]$ and $[\overline{\mathcal{H}}_g]$.

By using the previous result, we can immediately compute the class of the bielliptic locus in genus $2$. Let us take the boundary strata classes $d_{2,4}$ and $d_{3,3}$ as a basis for the $\Sym6$-invariant Picard group of $\Mmb06$. 
We have the following equalities (\cite[6.17, 6.18]{hamo}):
\begin{equation} \label{relationgenus0hyper}
\begin{cases} \phi^*\delta_0=2  d_{2,4}; \\ \phi^* \delta_1= \frac{1}{2}  d_{3,3}.
\end{cases}
\end{equation}
In view of \eqref{proofgenus2} and Proposition \ref{pullbackprop}, 
these relations give the expression for the class of the bielliptic locus 
in terms of $\delta_0$ and $\delta_1$ stated in Proposition \ref{prop2}.

\section{The class of the invariant locus in $\Mmb08$}
\label{invar}

In this section, we express the class $\invthree$ in terms of the generators of the invariant boundary strata classes in $\Mmb08$ (see Notation \ref{notation}):
\begin{equation} \label{basiscod2hyperell}
d_{5,1,2}, \quad d_{4,2,2}, \quad d_{4,1,3}, \quad d_{3,3,2}, \quad d_{3,2,3}, \quad d_{2,4,2}.
\end{equation}
We see from \cite[Theorem 5.9]{getzleroperads} that the $\Sym8$-invariant Chow group of codimension-$2$ classes in $\Mmb08$ has dimension $6$, so that these $6$ generators form a basis. We first compute the class of $\overline{I}_8$ in terms of boundary strata classes in $\Mmb08$.

Let $\pi_{1 2}$ and $\pi_{7 8}$ be the two forgetful maps from $\Mmb08$ to $\Mmb06$. The map $\pi_{1 2}$ forgets the marked points $1$ and $2$, and then renames $7,8$ to $1,2$. We have the equality \begin{equation} \label{naive} I_8= \pi_{1 2}^{-1}(I_6) \cap \pi_{7 8}^{-1}(I_6).\end{equation} Indeed, the right hand side contains curves admitting an involution $\sigma$ that permutes the last six marked points as $(34)(56)(78)$, and an involution $\tau$ that permutes the first six points as $(12)(34)(56)$. The composition $\sigma \circ \tau$ must be the identity, as it fixes $4$ points on a smooth genus $0$ curve, and this means that both $\sigma$ and $\tau$ do actually permute the eight marked points as $(12)(34)(56)(78)$.

Equality \eqref{naive} does not hold if one na\"ively puts closures on both sides; anyway, the argument above shows that $\overline{I}_8$ is an irreducible component of 
\begin{equation} \label{intersec}\pi_{1 2}^{-1}(\overline{I}_6) \cap \pi_{7 8}^{-1}(\overline{I}_6).\end{equation} 
We introduce the other components in $\Mmb08$ that are contained in this intersection.
\begin{enumerate}
\item Consider the locus $(1278|3456)$, whose generic element is a  curve with a node separating $\{1,2,7,8\}$ from $\{3,4,5,6\}$. The locus $\Div$ is the closure of the locus of curves in $(1278|3456)$ with the property that the node is invariant under the involution that exchanges the points $(34)$ and $(56)$.
\item Consider the $210$ codimension-$2$ irreducible components of $d_{2,4,2}$. Two of them have $\{1,2,7,8\}$ as marked points on the separating component and occur in \eqref{intersec}. We call the union of these strata $\typeI$.

\item Finally, consider the $280$ irreducible components of  $d_{3,2,3}$. Four of these boundary strata have the property that $\{3,4\}$ are the markings on the separating component, and $\{1,2\}$ are on two different components, and the same for $\{5,6\}$ and $\{7,8\}$. Four other ones come by exchanging $\{3,4\}$ and $\{5,6\}$ in the previous sentence. We call the union of these eight strata $\typeII$.
\end{enumerate}

\begin{figure}[ht]
\centering \footnotesize
\psfrag{1}{$1$} \psfrag{2}{$2$} \psfrag{3}{$3$} \psfrag{4}{$4$} \psfrag{5}{$5$} \psfrag{6}{$6$} \psfrag{7}{$7$} \psfrag{8}{$8$} \psfrag{I}{invariant}
\begin{tabular}{ccc}
\includegraphics[scale=0.4]{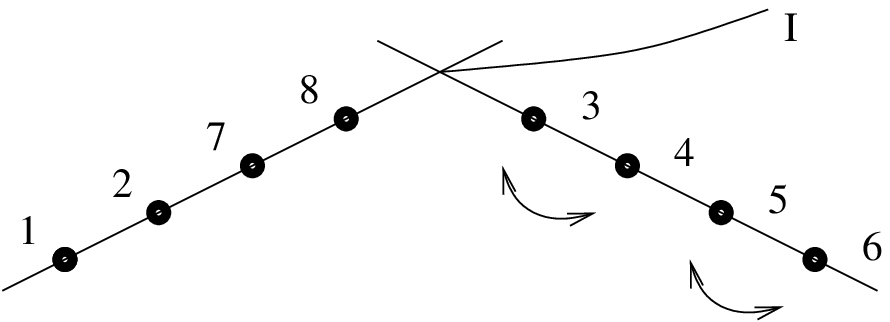} \ \ \ \ & \includegraphics[scale=0.4]{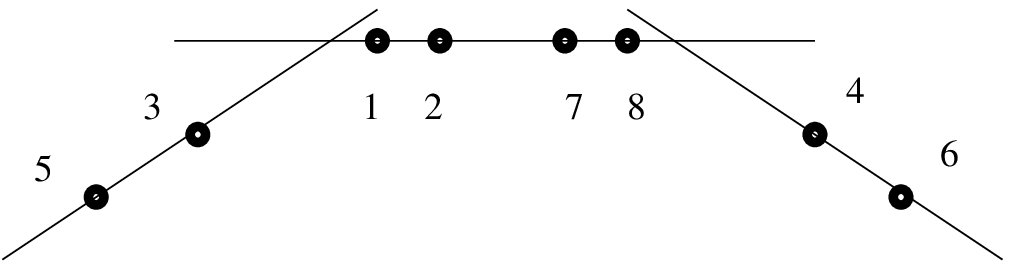} \ \ \ \ & \includegraphics[scale=0.4]{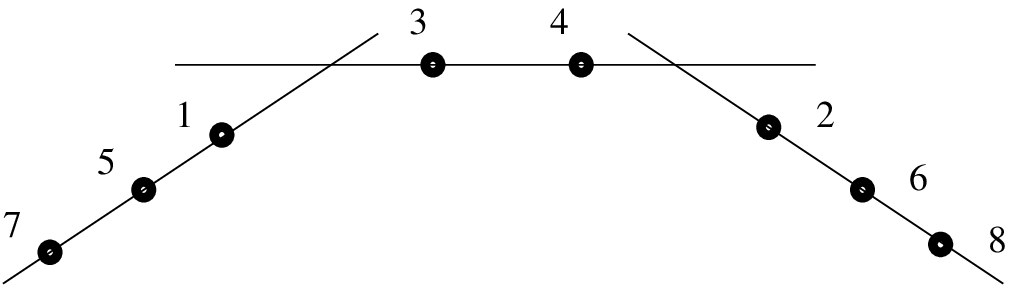}
\end{tabular}
\caption{The locus $\Div$, one component of $\typeI$, and one component of $\typeII$.}
\label{fig2}
\end{figure}
It is clear that these components are in the intersection \eqref{intersec}.
\begin{lemma}\label{vermAvermB} The following equality holds in $A^2(\Mmb08)$:
\begin{displaymath}\verma \cdot \vermb= \alpha [\overline{I}_8] + \beta [\Div]+ \gamma [\typeI] + \delta [{\typeII}].\end{displaymath}
\end{lemma}

We will prove at the end of this section that the coefficients $\alpha, \beta, \gamma$, and $\delta$ equal $1$. Assuming this for the moment, we have a way to express the class of $\overline{I}_8$ as an explicit linear combination of boundary strata classes in $A^2(\Mmb08)$. 
 
\begin{corollary} \label{classinvmz8} The class $\invthree$ equals:
\begin{displaymath}
\invthree= \frac{5}{2} d_{5,1,2}+ \frac{7}{4} d_{4,2,2}+ \frac{3}{4} d_{4,1,3}+ \frac{15}{4} d_{3,3,2}+ 3 d_{3,2,3} + \frac{3}{2} d_{2,4,2}.
\end{displaymath}
\end{corollary}

\begin{proof} In Lemma \ref{vermAvermB} we have expressed $[\overline{I}_8]$ in terms of the other classes; we will prove that the coefficients $\alpha, \beta, \gamma, \delta$ are all equal to $1$, see Lemmas \ref{alpha},  \ref{gammadelta}, and \ref{beta}. So let us say how one can express all other classes in terms of boundary.

\begin{enumerate}
\item By pulling back equality \eqref{verm}, it is not difficult to express $\verma$ and $\vermb$ in terms of boundary strata classes in $\Mmb08$. It is then lengthy but straightforward to express the product of the latter classes in terms of boundary. 

\item Let us study the class of $\Div$. On $\Mmb05$ with marked points $\{3,4,5,6, \bullet\}$, there is a divisor corresponding to the condition of $\bullet$ being fixed by the involution $(3 4)(5 6)$. The class of the latter divisor, after identifying $\Mmb05$ with the blow-up of $\mathbb{P}^2$ in four general points, is the proper transform (or pull-back) of the class of a line 
(namely, the line fixed pointwise by
the corresponding involution of $\mathbb{P}^2$ exchanging the four points
two-by-two).
Its class is therefore equal to $\psi_{\bullet}=(346) + (345) + (34)$. The class of $\Div$ is the push-forward of the class of this locus under the map that glues this $\Mmb05$ to the $\Mmb05$ with marked points $\{1,2,7,8, \star\}$:
\begin{displaymath} \begin{matrix}[\Div]=(1278|5|346)+(1278|6|345)+(1278|56|34). \\
\end{matrix}\end{displaymath}

\item The loci $\typeI$ and $\typeII$ are already boundary.

\end{enumerate}

Once this is settled, the class of $\overline{I}_8^{inv}$ can be computed in terms of the invariant classes 
$d_{5,1,2}$, $d_{4,3,3}$, $d_{4,1,3}$, $d_{3,3,2}$, $d_{3,2,3}$, 
and $d_{2,4,2}$ 
by symmetrizing, similarly to what was done in \eqref{proofgenus2}. The inverse image of the locus $I_8^{inv}$ in $\Mmb08$ is the union of $105$ irreducible components, each of them corresponding to an element of $\Sym8$ in the conjugacy class of $(12)(34)(56)(78)$. The numbers of irreducible components in $\Mmb08$ of the invariant loci are
$168$, $420$, $280$, $560$, $280$, and $210$, respectively.
\end{proof}

\noindent {\em Proof} (of Lemma \ref{vermAvermB}).
We want to show that
\begin{displaymath}\pi_{1 2}^{-1}(\overline{I}_6) \cap \pi_{7 8}^{-1}(\overline{I}_6)= \overline{I}_8 \cup \Div \cup \typeI \cup \typeII.\end{displaymath}
 That the right hand side is included in the left hand side is a straightforward check. To prove the other inclusion, we consider the stratification of $\Mmb08$ given by the number of nodes. 
As observed in \eqref{naive}, the restriction of the left hand side to the open part $\Mm08$ is precisely $I_8$. We conclude the proof by verifying two claims.  Recall that $\overline{I}_{n}$ has codimension $g-1$ for $n=2g+1$ or $2g+2$ (cf.~the beginning of Section \ref{loci}).

{\bf Claim 1.}
Among the boundary divisors of $\Mmb08$, only one has the property that $\pi_{1 2}^{-1}(\overline{I}_6)\cap\pi_{7 8}^{-1}(\overline{I}_6)$ cuts a codimension-$1$ locus on the open part of the divisor that parameterizes curves with precisely one singular point. This boundary divisor is $(1278|3456)$, and $\pi_{1 2}^{-1}(\overline{I}_6)\cap\pi_{7 8}^{-1}(\overline{I}_6)$ cuts out in it precisely the locus $\Div$.
(Clearly, the divisors of type $d_{5,3}$ do not contribute at all.
As to type $d_{6,2}$, the intersection is empty unless the $2$ points
form a pair $\{2i-1,2i\}$, in which case it is of codimension $2$.
For $d_{4,4}$, the components can be exchanged
or fixed. It suffices to consider $(1357|2468)$ in the former case.
Two different cross-ratios on the odd side must equal their counterparts
on the even side, yielding a codimension-$2$ locus. In the latter case,
it suffices to consider $(1234|5678)$, yielding the
codimension-$2$ locus where the nodal points on \emph{both} components must be
fixed by the induced involutions, and $(1278|3456)$, yielding the
codimension-$1$ locus where only the nodal point on the \emph{right} component
must be fixed.)

{\bf Claim 2.}
The only boundary strata classes of codimension $2$ in $\Mmb08$ that are included in $\pi_{1 2}^{-1}(\overline{I}_6)\cap\pi_{7 8}^{-1}(\overline{I}_6)$ are precisely those in $\typeI$ and $\typeII$.
(The images under both $\pi_{1 2}$ and $\pi_{7 8}$ of such boundary strata
must be boundary strata of $\Mmb06$ entirely contained in $\overline{I}_6$,
therefore necessarily of codimension $2$, of type $d_{2,2,2}$, with 
one of the pairs $\{2i-1,2i\}$ in the middle and the outer components
switched by the involution $(12)(34)(56)$. The claim is now easily verified
by adding the points $7$ and $8$ in all possible ways to such a stratum
and then applying the map $\pi_{1 2}$.)
\hfill\qed
\medskip

We now prove the equalities
\begin{displaymath}
\alpha=\beta=\gamma=\delta=1,
\end{displaymath}
for the coefficients that appear in Lemma \ref{vermAvermB}. This is needed to complete the proof of Corollary \ref{classinvmz8}. We have to compute multiplicities of intersections inside $\Mmb08$. 

\begin{lemma} \label{alpha} The coefficient $\alpha$ in Lemma \ref{vermAvermB} is $1$.
\end{lemma}
\begin{proof} We want to show that the intersection of $\verma$ and $\vermb$ has generically a reduced scheme structure. 

We recall a description of $\overline{I}_6$ due to Vermeire in \cite{vermeire} that uses Kapranov's description of $\Mmb06$. From Kapranov's construction, there is a blow-down map $\Mmb06 \to \mathbb{P}^3$. Vermeire has proved that $\overline{I}_6$ is the proper transform of the divisor $x_0 x_1 - x_2 x_3$.

Similarly, there is a blow-down map $\Mmb08 \to \mathbb{P}^5$. After restriction to a Zariski open subset $U$ of $\mathbb{P}^5$, $\verma$ is the proper transform of $x_0 x_1- x_2 x_3$ and $\vermb$ the proper transform of $x_0 x_1- x_4 x_5$. The two equations define a reduced subscheme of $U$.
\end{proof}

 For the remaining coefficients, we construct test surfaces for $\Mmb08$, and see that $\verma$ and $\vermb$ intersect transversely on them.

\begin{lemma} \label{gammadelta} The coefficients $\gamma$ and $\delta$ in Lemma \ref{vermAvermB} are both $1$.
\end{lemma}
\begin{proof}
We construct a test surface over $\mathbb{P}^1 \times \mathbb{P}^1$. The general fiber is a genus $0$ stable pointed curve with one node, which separates the odd markings from the even ones. The points $3$ and $4$ vary.

More precisely, we define two test $\mathbb{P}^1$'s in $\Mmb05$; their product is a test surface in $\Mmb05\times\Mmb05$ (with marked points $\{1,3,5,7, \star\}$, respectively $\{2,4,6,8, \bullet\}$) that yields a test surface in $\Mmb08$ by gluing the two last marked points.

In $\Mm04$ we fix the point $p$ that corresponds to:
\begin{displaymath}
5 \to 0, \quad \star \to \infty, \quad 1 \to 1, \quad 7 \to 2.
\end{displaymath}
The first $\mathbb{P}^1$ is obtained as the fiber over $p$ of the map $\Mmb05\to\Mmb04$ forgetting the last point (marked points $\{1,5,7,\star,3\}$), under the natural identification of this map with the universal curve. We call $\lambda$ the free parameter on this $\mathbb{P}^1$, corresponding to the marked point $3$.

The second $\mathbb{P}^1$ is constructed similarly, but starting from the point in $\Mm04$ that corresponds to:
\begin{displaymath}
6 \to 0, \quad \bullet \to \infty, \quad 2 \to 1, \quad 8 \to 3. 
\end{displaymath}
We call $\mu$ the free parameter on the second $\mathbb{P}^1$, corresponding to the marked point $4$.

The two divisors $\verma$ and $\vermb$ define two divisors on the test surface. The divisor $\verma$ imposes the condition that the quadruples $5,\star, 7, 3$ and $6, \bullet, 8, 4$ define the same point on $\Mmb04$. Thus it is given by the equation
$3\lambda=2\mu$ on $\mathbb{P}^1_{\lambda} \times \mathbb{P}^1_{\mu}$.
The divisor $\vermb$ imposes instead that the quadruples $5, \star, 1, 3$ and $6, \bullet, 2, 4$ define the same point on $\Mmb04$. Thus it corresponds to the equation
$\lambda= \mu$ on $\mathbb{P}^1_{\lambda} \times \mathbb{P}^1_{\mu}$.
The set-theoretic intersection is therefore in $\lambda=\mu=0$ and $\lambda=\mu=\infty$. The solution $\lambda=\mu=0$ corresponds to a curve 
that is a chain of four components, with distribution $(35|17|28|46)$
of the marked points ($\typeI$). The second solution corresponds also
to a chain of four components, with distribution $(157|3|4|268)$ of the marked points ($\typeII$).
The fact that the two equations have bidegree $(1,1)$ is enough to establish that both intersection multiplicities are $1$.
\end{proof}

\begin{lemma} \label{beta} The coefficient $\beta$ in Lemma \ref{vermAvermB} is $1$.
\end{lemma}
\begin{proof}
We construct the following test surface. Fix $4$ distinct points $(3,4,5,6)$ on a smooth genus $0$ curve $C$, and let two points $1, \bullet$ vary on it. This defines a test surface for $\Mmb06$ (markings $\{1,3,4,5,6, \bullet\})$. This also gives a test surface for $\Mmb08$ once a choice of $4$ distinct points  $(2,7,8, \star)$ is fixed on a smooth genus $0$ curve, by gluing $\bullet$ with $\star$.

We fix an isomorphism of $C$ with $\mathbb{P}^1$ in such a way that \begin{displaymath}3 \to 0 \quad 4 \to \infty, \quad 5 \to 1,\quad 6 \to 4, \quad \bullet \to \lambda, \quad 1 \to \mu.\end{displaymath} On this test surface, with this choice of coordinates, $\verma$ is given by the equation $\lambda^2 =4$ (the involution
$x\mapsto4/x$ must fix $\lambda$)
and $\vermb$ is given by $\lambda\mu=4$ (the same involution
must switch $\lambda$ and $\mu$).
They clearly intersect transversely in 
the points $(2,2)$ and $(-2,-2)$, both belonging to $\Div$.
\end{proof}

\section{Pulling back from $\Mb3$ to the hyperelliptic locus}
\label{Pulling}

In this section we study the linear map
\begin{equation} \label{pullcod2}
\phi^*\colon A^2(\Mb3) \to A^2(\overline{\mathcal{H}}_3^{adm})^{\Sym8}. 
\end{equation}
We have fixed \eqref{basiscod2} as the basis in the domain, and  
\eqref{basiscod2hyperell} as the basis in the image.

In the lemma below, we will need explicit expressions for certain invariant tautological classes in $\Mmb08$ in terms of invariant boundary strata classes. The computations necessary for this were done using \cite{fabermaple}. 

Recall the Arbarello-Cornalba ${\kappa}$ classes:
\begin{displaymath}
{\kappa_i}:= \pi_*\left(c_1(\omega_{\pi}(D))^{i+1}\right)
\end{displaymath}
where $\pi$ is the universal curve over $\Mmb08$, $\omega_{\pi}$ is the relative dualizing sheaf, and $D$ is the divisor corresponding to the $8$ disjoint sections in the universal curve. Another useful invariant class will be the codimension-$j$ class $\tilde{\psi}_j:= \sum_{i=1}^8 \psi_i^j$, where, as usual, $\psi_i$ is the first Chern
class of the cotangent line bundle at the $i$th marked point.

In codimension $1$ we have (observe that ${\kappa_1}= \tilde{\psi_1}- d_{2,6} - d_{3,5} - d_{4,4}$): 
\begin{equation} \begin{cases}\label{cod1} {\kappa_1}=& \frac{5}{7} d_{2,6} + \frac{8}{7} d_{3,5} + \frac{9}{7} d_{4,4}, \\
\tilde{\psi_1}=& \frac{12}{7} d_{2,6} + \frac{15}{7} d_{3,5} + \frac{16}{7} d_{4,4}.
\end{cases}
\end{equation}
In codimension $2$, we obtain:
\begin{equation} \label{cod2}
\begin{cases}
{\kappa_2}= & \frac{1}{7} d_{5,1,2} + \frac{1}{7} d_{4,2,2} + \frac{6}{35} d_{4,1,3} + \frac{1}{10} d_{3,3,2} + \frac{6}{35} d_{3,2,3} + \frac{1}{21} d_{2,4,2}, \\
\tilde{\psi_2}= &\frac{11}{21} d_{5,1,2} + \frac{16}{35} d_{4,2,2} + \frac{3}{7} d_{4,1,3} + \frac{3}{10} d_{3,3,2} + \frac{3}{7} d_{3,2,3} + \frac{16}{105} d_{2,4,2}. \\
\end{cases}
\end{equation}
Finally, we can express the products of the invariant codimension-$1$ classes in terms of the invariant codimension-$2$ classes:
\begin{equation} \label{cod1cod2} \begin{cases}
d_{2,6}^2&=  - \frac{2}{3} d_{5,1,2}- \frac{2}{5} d_{4,2,2} - \frac{1}{5} d_{3,3,2} + \frac{28}{15} d_{2,4,2} ,\\
d_{2,6} d_{3,5}&=  d_{5,1,2} + d_{3,3,2} ,\\
d_{2,6} d_{4,4}&=  d_{4,2,2} ,\\
d_{3,5}^2&= -\frac{1}{3} d_{5,1,2} - \frac{3}{5} d_{4,1,3} - \frac{1}{10} d_{3,3,2} + \frac{7}{5} d_{3,2,3} ,\\
d_{3,5} d_{4,4}&= d_{4,1,3},\\
d_{4,4}^2&=  - \frac{1}{6} d_{4,2,2} - \frac{1}{2} d_{4,1,3}.\\
\end{cases}
\end{equation}
We are now in a position to compute the matrix associated to $\phi^*$. 
\begin{lemma} \label{matrix} The following $7 \times 6$ matrix is associated to $\phi^*$ in the bases \eqref{basiscod2} and \eqref{basiscod2hyperell}$\colon$
\begin{displaymath} 
\renewcommand{\arraystretch}{1.2}
\begin{pmatrix}
\frac{1}{42} &\frac{19}{210}&\frac{1}{35}&\frac{1}{20}&
      \frac{3}{35}&\frac{1}{35} \\
      
      0&\frac{11}{15}&0&\frac{1}{5}&0& \frac{4}{5}\\
      
      \frac{1}{12}&0&\frac{1}{10}&\frac{1}{10}&\frac{1}{10}&0\\
      
      -\frac{8}{3}&\frac{86}{15}&-2&-\frac{4}{5}&0 &\frac{112}{15}\\
      
      1& 0 & 1 & 1 & 0 &0 \\
      
      -\frac{1}{12}&0&- \frac{3}{20} & - \frac{1}{40}& \frac{7}{20} & 0 \\
      
      \frac{13}{84}& \frac{6}{35} & \frac{33}{140} & \frac{1}{8} & \frac{33}{140} & \frac{2}{35}
\end{pmatrix}
\end{displaymath}
\end{lemma}

\begin{proof} Let us first study the pull-back $\phi^*$ at the level of codimension-$1$ classes. We fix $d_{2,6}$, $d_{3,5}$ and $d_{4,4}$ as a basis of $A^1(\overline{\mathcal{M}}_{0,8})^{\Sym8}$. We have:
\begin{equation}\begin{cases} \label{first}
\phi^* \lambda= & \frac{3}{14} d_{2,6}+ \frac{1}{7} d_{3,5} + \frac{2}{7} d_{4,4},\\
\phi^* \delta_0= & 2 d_{2,6} + 2 d_{4,4} ,\\
\phi^* \delta_1= & \frac{1}{2} d_{3,5} .\\
\end{cases}
\end{equation}
The last two equalities are well-known (see \cite[6.17, 6.18]{hamo}), while the first is obtained from the standard equation $12 \lambda= \kappa_1+ \delta_0 + \delta_1$. Indeed, from \cite[p.~234]{fpmmj}, we have $\phi^*(\kappa_1)= 2 {\kappa_1}- \frac{1}{2} \tilde{\psi_1}$, 
which gives the result by means of \eqref{cod1}.

From this and from \eqref{cod1cod2}, it is immediate to compute the pull-back of the basis \eqref{basiscod2} in terms of the basis \eqref{basiscod2hyperell}. We need only compute $\phi^*(\kappa_2)$, and this can be done by means of the equality $\phi^*(\kappa_2)= 2 {\kappa_2} - \frac{1}{4} \tilde{\psi_2}$ (\cite[p.~234]{fpmmj}). The two terms on the right are expressed in the basis \eqref{basiscod2hyperell} with the use of \eqref{cod2}.
\end{proof}

By putting together Corollary \ref{classinvmz8}, Proposition \ref{pullbackprop}, and Lemma \ref{matrix}, we have an explicit expression for $\bithree$ in the basis \eqref{basiscod2} up to two parameters. For example, the coordinates for $\bithree$ in the basis \eqref{basiscod2} can be written in terms of the coefficient of $\delta_0^2$ (that we call $d$) and of $\epsilon$, the parameter introduced in Proposition \ref{pullbackprop}:
\begin{equation} \label{partialbiell}
\left(\frac{459+560d \epsilon}{6 \epsilon}, \frac{-18-58d\epsilon}{3\epsilon}, \frac{-117-136d\epsilon}{3\epsilon}, d, \frac{18+14d\epsilon}{3\epsilon}, \frac{99+32d\epsilon}{6\epsilon}, -\frac{9}{2\epsilon}\right).
\end{equation}

\section{Test surfaces}

\label{testsurfaces}
In this section we study 
three families of genus $3$ stable curves over surfaces $\Sigma_8$, $\Sigma_1$, 
and $\Sigma_2$. The last 
two of these test surfaces for $\Mb3$ were studied earlier by the first author (see \cite[Section 2]{faber} for details); the 
new test surface $\Sigma_8$ is closely related to $\Sigma_2$. We will be able to count the number of bielliptic curves on each of these families by means of elementary considerations. This will provide the following information:
\begin{enumerate}
\item The computation of the number of bielliptic curves on $\Sigma_8$ will prove that the coefficient $\epsilon$ in Proposition \ref{pullbackprop} equals $1$. 
\item The computation of the number of bielliptic curves on $\Sigma_1$ will complete the proof of our main result: Theorem \ref{bithree}.
\item The computation of the number of bielliptic curves on $\Sigma_2$ gives us a consistency check on Corollary \ref{classinvmz8}.
\end{enumerate} 
We use $\Sigma_8$ (instead of the surface $\Sigma_5$ discussed below)
to avoid excess intersections. 

\begin{remark} \label{transbiell} For a sufficiently general family of stable curves, the intersection between the loci of admissible (double) covers and the locus of singular curves is transversal. This is a consequence of the fact that admissible covers are smoothable (cf.~\cite[Thm.~3.160 and p.~186]{hamo}).
In particular, we will use the transversality of the bielliptic locus and the locus parameterizing singular curves. Compare \cite[Lemma 11.6.15] {acg}, where the authors work out this transversality result for the hyperelliptic locus and the boundary of $\Mb{g}$.
\end{remark}

The new test surface $\Sigma_8$ is the product $E \times \mathbb{P}^1$,
where $(E,0)$ is an elliptic curve on which a general point $a$ is fixed.
Given a simple pencil $\mathcal{P}$ of elliptic curves over $\mathbb{P}^1$
(with $12$ degenerate fibers), the fiber over a point $(e,q)$ with 
$e\notin\{0,a\}$
is obtained by gluing $E$ to itself at the points
$a$, $e$ and to $\mathcal{P}_q$ at the respective origins. When $e=a$,
a degenerate elliptic curve is glued to $E$ at $a$. When $e=0$, the fiber consists
of the $3$ components $E$, $\mathbb{P}^1$, and $\mathcal{P}_q$, with the
identifications $0_E\sim0_{\mathbb{P}^1}$, $a_E\sim \infty_{\mathbb{P}^1}$,
and $1_{\mathbb{P}^1}\sim 0_{\mathcal{P}_q}$.

It is not difficult to check that the only fibers in the family admitting an
admissible bielliptic involution are found when $e=a$ and $\mathcal{P}_q$
is degenerate ($12$ fibers). These curves are all isomorphic and have $8$
automorphisms; exactly $2$ automorphisms are admissible bielliptic involutions.
The value of $\bithree$ on $\Sigma_8$ is therefore at least $24$ and it equals
$24$ if at each of the $12$ points the intersection of $\Sigma_8$ with each
of the two branches of $\overline{\mathcal{B}}_3$ is transversal.

To construct the subfamily where $q$ equals a given point of $\mathbb{P}^1$,
one blows up $E\times E$ at $(0,0)$ and $(a,a)$
and first glues the blown up surface to itself by identifying the proper transforms
of the diagonal and the $a$-section and then glues the obtained surface
to $E\times \mathcal{P}_q$ by identifying the proper transform of the
$0$-section with the $0$-section of $E\times \mathcal{P}_q$. The total family
is obtained by varying $q$.
Up to numerical equivalence, we easily obtain $\lambda=(0,1)$,
$\delta_0=(-2,12)$, and $\delta_1=(0,-1)$. We find the following intersection
numbers:
\begin{displaymath}\lambda^2=\lambda \delta_1= \delta_1^2=\kappa_2=0, \quad 
\lambda\delta_0=-2, \quad \delta_0^2=-48, \quad \delta_0 \delta_1=2.\end{displaymath}
(One obtains $\kappa_2=0$ just as for $\Sigma_2$, cf.~\cite[Section 2.2]{faber}.)

Formula (\ref{partialbiell}) gives then 
$24/\epsilon$
for the value of $\bithree$ on $\Sigma_8$. Since $\epsilon$ is an integer,
we immediately conclude that it equals $1$ and that the transversality
alluded to above indeed holds. In fact, the transversality of $\Sigma_8$ 
to the branches of $\overline{\mathcal{B}}_3$ can be seen easily, which
avoids the use of the integrality of $\epsilon$. (The essential point
is that there are two different simultaneous infinitesimal smoothings
of the two nodes of type $\delta_1$, differing by composition with
$-{\rm Id}$ and both preserving exactly one bielliptic involution. From this,
one easily deduces the transversality of the tangent space to $\Sigma_8$
to the two tangent spaces to the branches of $\overline{\mathcal{B}}_3$.)

Therefore, the coordinates for $\bithree$ in the basis \eqref{basiscod2} become:
\begin{equation} \label{lesspartialbiell}
\left(\frac{459+560d }{6 }, \frac{-18-58d}{3}, \frac{-117-136d}{3}, d, \frac{18+14d}{3}, \frac{99+32d}{6}, -\frac{9}{2}\right).
\end{equation}

For the first test surface (\cite[Section 2.1]{faber}), we consider a general curve $C$ of genus $2$. Over the surface $\Sigma_1=C \times C$, there is a family whose fiber over $(p,q)$ is obtained by gluing on $C$ the two points $p$ and $q$. The fibers admit a (unique) admissible bielliptic involution when $p$ and $q$ are distinct Weierstrass points of $C$, thus there are $30=6 \times 5$ such fibers. 
By Remark \ref{transbiell} above, the intersection of $\overline{\mathcal{B}}_3$
and $\Sigma_1$ is transversal (note that at each of the $30$ points
the tangent space to $\Mb3$ equals the direct sum of the tangent space
to $\overline{\mathcal{B}}_3$ and the tangent space to the surface).
So the value of the class $\bithree$ restricted to $\Sigma_1$ equals $30$. We read in \cite[Proposition 2.1]{faber} that on this surface the following equalities hold: \begin{displaymath}\lambda^2=\lambda \delta_0= \lambda \delta_1=\delta_0 \delta_1=0, \quad \delta_0^2=16, \quad \delta_1^2=-2, \quad \kappa_2=2.\end{displaymath}
Therefore, equation \eqref{lesspartialbiell} gives the last parameter: $d=\frac{27}{2}$. Theorem \ref{bithree} is proved.

The second test surface (\cite[Section 2.2]{faber}) is the product $C \times \mathbb{P}^1$, where $C$ is a general genus $2$ curve. 
Given a simple pencil of elliptic curves over $\mathbb{P}^1$, the fiber over a point $(p,q)$ is obtained by identifying the point $p$ of $C$ with the origin of the elliptic curve over $q$. 
No 
fiber admits an admissible bielliptic involution. We read in \cite[Proposition 2.2]{faber} that on this surface the following equalities hold: \begin{displaymath}\lambda^2=\lambda \delta_0= \delta_0^2=\kappa_2=0, \quad \lambda \delta_1=-2, \quad \delta_0 \delta_1=-24, \quad \delta_1^2=4.\end{displaymath}
These numbers, substituted in equation \eqref{partialbiell}, give a nontrivial consistency check of Corollary \ref{classinvmz8} and of Lemma \ref{matrix}.

Note that after the results of \cite[Section 2]{faber}, it is equivalent to know the class of the bielliptic locus in the basis \eqref{basiscod2}, and to know the restriction of the bielliptic class to the seven test surfaces. While we can compute $\bithree$ on $\Sigma_1, \ldots, \Sigma_5$, we do not know a direct way to compute it on $\Sigma_6$ and $\Sigma_7$.

After Theorem \ref{bithree} and \cite[Section 2]{faber} however, 
we immediately obtain:
\begin{displaymath}
\bithree|_{\Sigma_3}=-24, \quad \bithree|_{\Sigma_4}=33, 
\quad \bithree|_{\Sigma_5}=9,
\quad \bithree|_{\Sigma_6}=225, \quad \bithree|_{\Sigma_7}=675.
\end{displaymath}

A direct computation of $\bithree$ on
the fifth test surface (\cite[Section 2.5]{faber}) runs as follows.
We consider two elliptic curves $(E,0)$ and $(F,0)$. Over the surface $\Sigma_5=E \times F$, there is a family of genus $3$ curves, whose fiber over $(e,f)$ with $e\neq0$ and $f\neq0$ is obtained by gluing $E$ and $F$ at the origins and at $e,f$. For a fiber to admit an admissible bielliptic involution, $e$ or $f$ must be a point of order $2$.
The bielliptic locus $\overline{\mathcal{B}}_3$ thus has excess intersection with $\Sigma_5$ along $3$ horizontal and $3$ vertical curves. As it turns out, the \emph{excess} intersection does not contribute to the intersection number. However, the fibers over the $9$ points of intersection carry a \emph{third} admissible bielliptic involution, yielding $9$ for the value of $\bithree$ on $\Sigma_5$.

Let us also consider the sixth test surface $\Sigma_6$ (\cite[Section 2.6]{faber}). It is obtained by applying stable reduction to a general linear $\mathbb{P}^2$ inside the $\mathbb{P}^{14}$ of plane quartics. Since the bielliptic locus has codimension $2$ and the locus of singular curves has codimension $1$, for such a $\mathbb{P}^2$ 
the bielliptic points correspond to smooth bielliptic curves.
Again by genericity, we have that the number $225$ enumerates the number of smooth bielliptic curves on the linear $\mathbb{P}^2$, and that the codimension-$2$ bielliptic locus in $\mathbb{P}^{14}$ has degree $225$.




\begin{thebibliography}{ACG11}

\bibitem [ACV03] {acv} Dan Abramovich, Alessio Corti, Angelo Vistoli, \emph{Twisted bundles and admissible covers}, Special issue in honor of Steven L. Kleiman. Comm. Algebra 31 (2003), no. 8, 3547--3618.

\bibitem [ACG11] {acg} Enrico Arbarello, Maurizio Cornalba, Phillip Griffiths, \emph{Geometry of algebraic curves}, Vol. II, 
Grundlehren der Mathematischen Wissenschaften, 268. Springer, Heidelberg, 2011.

\bibitem [Fa90] {faber} Carel Faber, \emph{Chow rings of moduli spaces of curves. I. The Chow ring of $\overline{ M}_3$}, Ann. of Math. (2) 132 (1990), no. 2, 331--419.

\bibitem [Fa96] {faberbanach} Carel Faber, \emph{Intersection-theoretical computations on $\overline { M}_g$}, Parameter spaces (Warsaw, 1994), 71--81, Banach Center Publ., 36, Polish Acad. Sci., Warsaw, 1996.

\bibitem [FaMa] {fabermaple} Carel Faber, {Maple program to compute intersection numbers on $\overline{M}_{0,n}$}, available upon request.

\bibitem [FP00] {fpmmj} Carel Faber, Rahul Pandharipande, \emph{Logarithmic series and Hodge integrals in the tautological ring}, with an appendix by Don Zagier. Dedicated to William Fulton on the occasion of his 60th birthday. Michigan Math. J. 48 (2000), 215--252. 

 \bibitem[Ge94] {getzleroperads} Ezra Getzler, \emph{Operads and moduli of genus $0$ Riemann surfaces}, The moduli space of curves (Texel Island, 1994), 199--230, Progr. Math., 129, Birkh\"auser Boston, Boston, MA, 1995.
 
 \bibitem [HMo98] {hamo} Joe Harris, Ian Morrison, \emph{Moduli of Curves}, Springer, New York, 1998.

\bibitem [HMu82] {hamu} Joe Harris, David Mumford, \emph{On the Kodaira dimension of the moduli space of curves},  Invent. Math. 67 (1982), no. 1, 23--88. With an appendix by William Fulton.
 
\bibitem[Ka93] {kapranov} Mikhail Kapranov, \emph{Veronese curves and Grothendieck-Knudsen moduli space $\overline M_{0,n}$}, J. Algebraic Geom. 2 (1993), no. 2, 239--262. 

\bibitem[Mu83] {M-Enum} David Mumford, \emph{Towards an enumerative geometry of the moduli space of curves}, Arithmetic and geometry, Vol. II, 271--328, Progr. Math., 36, Birkh\"auser Boston, Boston, MA, 1983.

\bibitem [Ve02] {vermeire} Peter Vermeire, \emph{A counterexample to Fulton's conjecture on $\overline M_{0,n}$}, J. Algebra 248 (2002), no. 2, 780--784.


\end{thebibliography}
\end{document}